\documentclass[11pt]{amsart}
\usepackage[utf8]{inputenc}
\usepackage{graphicx,color}
\usepackage{amssymb}
\usepackage{booktabs}
\usepackage{hyperref}
\usepackage[T1]{fontenc}

\usepackage{lineno}

\usepackage{tikz}

\usetikzlibrary{shapes}

\newtheorem{thm}{Theorem}[section]
\newtheorem{prop}[thm]{Proposition}
\newtheorem{lem}[thm]{Lemma}
\newtheorem{cor}[thm]{Corollary}

\newtheorem{qst}[thm]{Question}

\theoremstyle{definition}

\theoremstyle{remark}
\newtheorem{rmk}[thm]{Remark}
\newtheorem{eg}[thm]{Example}

\newcommand{\Z}{\mathbb Z}

\newcommand{\C}{\mathbb C}

\renewcommand{\hat}{\protect\widehat}

\newcommand{\rank}{\mathop{\mathrm{rank}}\nolimits}

\newcommand{\cx}{\mathop{\mathrm{cx}}\nolimits}

\newcommand{\codim}{\mathop{\mathrm{codim}}\nolimits}

\newcommand{\depth}{\mathop{\mathrm{depth}}\nolimits}

\newcommand{\K}{\mathbb{K}}
\newcommand{\supp}{\mathop{\mathrm{Supp}}\nolimits}

\newcommand{\HF}{\mathop{\mathrm{HF}}\nolimits}
\newcommand{\HS}{\mathop{\mathrm{HS}}\nolimits}

\newcommand{\tensor}{\otimes}
\newcommand{\tor}{\mathop{\mathrm{Tor}}\nolimits}

\title[Koszul Quotients of Exterior Algebras]{G-Quadratic, LG-Quadratic, and Koszul Quotients of Exterior Algebras}
\author[J. McCullough]{Jason McCullough}
\address{Iowa State University, Department of Mathematics, Ames, IA, USA}
\email{jmccullo@iastate.edu}
\author[Z. Mere]{Zachary Mere}
\address{Iowa State University, Department of Mathematics, Ames, IA, USA}
\email{zwmere@iastate.edu}

\begin{document}



\begin{abstract}
This paper introduces the study of LG-quadratic quotients of exterior algebras, showing that they are Koszul, as in the commutative case.  We construct an example of an LG-quadratic algebra that is not G-quadratic and another example that is Koszul but not LG-quadratic.  This is only the second known Koszul algebra that is not LG-quadratic and the first that is noncommutative.
\end{abstract}

\maketitle

\section{Introduction}

Let $\K$ be a field and let $E = \bigwedge_\K \langle e_1,\ldots,e_n \rangle$ denote an exterior algebra over $\K$ on $n$ variables.  The purpose of this paper is to investigate the Koszul and G-quadratic properties of quotients of $E$.  In the commutative setting, we have the following implications:
\[\text{quadratic GB} \Rightarrow \text{G-quadratic} \Rightarrow \text{LG-quadratic} \Rightarrow \text{Koszul} \Rightarrow \text{quadratic}.\]
Each of these implications is strict; see \cite[p. 292]{CDR13}.  The third implication is particularly interesting as there is only one known commutative Koszul algebra that is not LG-quadratic due to Conca \cite[Example 1.20]{Conca}.  

Over an exterior algebra, we show that the same implications hold and that all of them are strict.  In particular, we introduce the notion of LG-quadratic quotients of an exterior algebra and prove that they are Koszul (Theorem~\ref{LG}).  We construct an LG-quadratic quotient of an exterior algebra that is not G-quadratic (Theorem~\ref{LGnotG}), thus answering a question of Thieu \cite[Example 5.2.2]{Thieu13}.  We also construct a Koszul quotient of an exterior algebra that is not LG-quadratic (Theorem~\ref{KoszulNonLG}).  This is then only the second such algebra and the first that is noncommutative.  It is also non-obstructed, in the language of Conca \cite[Example 1.21]{Conca}, meaning that there are algebras defined by quadratic monomials with the same Hilbert function.  This is one reason for the difficulty in proving that our construction is indeed not LG-quadratic.

One of the primary motivations for this paper comes from the theory of Orlik-Solomon algebras of complex hyperplane arrangements, which are quotients of exterior algebras.  It follows from work of Peeva \cite{Peeva03} and Bj\"orner and Zieglier \cite{BZ91} that a hyperplane arrangement is supersolvable if and only if its Orlik-Solomon algebra has a defining ideal with a quadratic Gr\"obner basis (in a particular system of coordinates).  It is an open question whether there are Koszul Orlik-Solomon algebras arising from non-supersolvable arrangements; see \cite[Problem 82]{Schenck12} and \cite[p. 487]{SY97}.  The two notions are equivalent for hypersolvable arrangements \cite{JP98},  graphic arrangements \cite{SS02}, Dirichlet arrangements \cite{Lutz19}, root ideal arrangements \cite{Hultman16}, and others \cite{VR13}.

The rest of this paper is organized as follows.  Section~\ref{back} collects the necessary notation and background results. In Section~\ref{lg} we show that LG-quadratic quotients of exterior algebras are Koszul.  In Section~\ref{depth} we construct an LG-quadratic algebra (which is necessarily Koszul) that is not G-quadratic.  We first consider the depth of edge ideals associated to graphs and construct our example as a quotient of an algebra defined by quadratic monomials.  In Section~\ref{nonLG} we show that six generic quadrics in six exterior variables define a Koszul algebra that is not LG-quadratic.  A final Section~\ref{examples} collects some remaining relevant examples.

\section{Background}\label{back}

In this section we fix notation and recall some results from the literature.  
Throughout this paper, let $\K$ denote a field and let $V$ be a $\K$-vector space of dimension $n$.  Let $e_1,\ldots,e_n$ denote a fixed basis of $V$ and let $E = \bigwedge (V) = \bigwedge_\K\langle e_1,\ldots, e_n \rangle$ denote the exterior algebra of $V$.  
We view $E = \oplus_{i = 0}^n E_i$ as a skew-commutative, graded ring with $\deg(e_i) = 1$ for $1 \le i \le n$ and with $E_i$ denoting the $\K$-span of the degree $i$ monomials in $E$.  An $E$-module is graded if there is a $\K$-vector space direct sum decomposition $M = \oplus_i M_i$ such that $E_i \cdot M_j \subseteq M_{i+j}$ for all integers $i$ and $j$.  A \textit{monomial} in $E$ is an expression of the form $e_{i_1} \wedge e_{i_2} \wedge  \cdots \wedge e_{i_t}$; we usually omit the wedge products and write $e_{i_1}e_{i_2}\cdots e_{i_t}$ instead.  An ideal generated by monomials is a monomial ideal.  


\subsection{Free resolutions, Hilbert series, Poincare series}
Let $R$ be a $\Z$-graded $\K$-algebra, and
let $M$ be a finitely generated, graded $R$-module.  The Hilbert function of $M$ is  $\HF(M,i) = \dim_\K(M_i)$.  Its generating series is the Hilbert series $\HS_M(t) = \sum_{i} HF(M,i) t^i$.  We recall that $\HS_E(t) = (1+t)^n$.

The module $M$ has a minimal graded free resolution $\mathbf{F}_\bullet$ over $R$, where $F_i = \bigoplus_j R(-j)^{\beta_{i,j}}$.  Here $R(-j)$ denotes a rank-one graded free $R$-module with $R(-j)_i = R_{i-j}$.  The numbers $\beta^R_{i,j}(M)$ are the graded Betti numbers of $M$ and are invariants of $M$.  The $i$th total Betti number of $M$ is $\beta^R_i(M) = \sum_j \beta^R_{i,j}(M)$.  We may alternatively compute the Betti numbers of $M$ as $\beta^R_{i,j}(M) = \dim_\K \tor_i^R(M,\K)_j$ and $\beta^R_i(M) = \dim_\K \tor_i^R(M,\K)$.  Note that most resolutions over $E$ are infinite.  The generating series of the total Betti numbers is the Poincare series of $M$ denoted $P^R_M(t) = \sum_{i \ge 0} \beta^R_i(M) t^i$.  

\subsection{Depth, regular elements, and singular varieties}
Let $M$ be a graded $E$-module. An element $\ell \in E_1$ is \textit{regular} on $M$ if \[\{m\in M\mid \ell m=0\}=\ell M.\] Equivalently, $\ell \in E_1$ is regular if \[ M(-1)\xrightarrow{\ell \cdot} M\xrightarrow{\ell \cdot} M(1)\] is an exact sequence of graded $E$-modules.

If $\ell \in E_1$ is not $M$-regular, we say that $\ell$ is $M$-singular. The set of $M$-singular elements in $E_1$, denoted $V_E(M)$, is called the \textit{singular variety} of $M$.  The following theorem of Aramova, Avramov, and Herzog will be useful when attempting to compute singular varieties (termed rank varieties in \cite{AAH00}.)

\begin{thm}[{\cite[Theorem 3.1]{AAH00}}]\label{thmAAH} If $\K$ is a field, then the singular varieties of finite graded $E$-modules $M$ and $N$ satisfy the following properties:
\begin{enumerate}
    \item $V_E(M)$ is  a union of linear subspaces in $E_1$.
    \item If $M\subseteq N$, then each  of $V_E(M), V_E(N), V_E(N/M)$ is contained in the union of the other two.
\end{enumerate}
\end{thm}

A sequence of elements $\ell_1,\ldots,\ell_d$ is a regular sequence on $M$ if $\ell_i$ is regular on $M/(\ell_1,\ldots,\ell_{i-1})M$ for $i = 1,\ldots,d$.  The maximal length of a regular sequence on $M$ is the depth of $M$ denoted $\mathrm{depth}_E(M).$  Note that $\mathrm{depth}_E(M) > 0$ if and only if $V_E(M) \neq E_1$.  Moreover, $HS_M(t)$ is divisible by $(1+t)^{\depth_E(M)}$.

The complexity of an $E$-module $M$ is
\[\cx_E(M) = \inf\{ c \in \Z\,|\,\beta_i^E(M) \le \alpha i^{c-1} \text{ for some } \alpha \in \mathbb{R}\}.\]
There is the following analogue of the Auslander-Buchsbaum theorem connecting depth and complexity.
\begin{thm}[{\cite[Theorem 3.2]{AAH00}}]\label{ab} If $\K$ is an infinite field and $M$ is a finitely generated $E$-module, then every maximal regular sequence has $\depth_E(M)$ elements and
\[\depth_E(M) + \cx_E(M) = n.\]
\end{thm}

\subsection{Gr\"obner bases}\label{gb}

Fix a monomial order $<$ on $E$ and a graded ideal $I \subseteq E$.  Denote the leading term of $f \in E$ by $LT(f)$ and the set of monomials of $f$ with nonzero coefficients by $\supp(f)$, called the support of $f$.  The initial ideal of $I$ is $in_<(I) = (LT(f):f \in I)$.
Then $I$ has a unique reduced Gr\"obner basis $g_1,\ldots,g_t \in I$; in other words, $(LT(g_1),\ldots,LT(g_t)) = in_<(I)$, each leading coefficient is $1$, and no monomial in $\supp(g_i)$ is divisible by $LT(g_j)$ for any $1 \le j \le t$.  There is a Buchberger algorithm for computing Gr\"obner bases and we refer the reader to \cite{HH11} for more details.  In particular, we use that $HS_{E/I}(t) = HS_{E/in_<(I)}(t)$ \cite[Corollary 1.2]{AHH97}.  

Just as with initial ideals over a polynomial ring, the graded Betti numbers of a graded ideal over an exterior algebra can only go up when one passes to the initial ideal.  One can then prove the analogous result for depth.

\begin{prop} Let $E$ be an exterior algebra and let $I$ be a graded ideal in $E$.  Then $\depth_E(E/in_<(I))\leq \depth_E(E/I)$. \end{prop}  

\begin{proof} By \cite[Proposition 1.8]{AHH97}, $\beta_{ij}(E/I)\leq \beta_{ij}(E/in_<(I))$ for all $i,j$. Hence \begin{align*} \cx_E(E/I)&=\inf \{c\in \mathbb Z\mid \beta_i^E(E/I)\leq \alpha i^{c-1} \text{ for some } \alpha \in \mathbb R\}\\&\leq \inf \{c\in \mathbb Z\mid \beta_i^E(E/in_<(I))\leq \alpha i^{c-1} \text{ for some } \alpha \in \mathbb R\}\\&=\cx_E(E/in_<(I)). \end{align*} The claim then follows from Theorem~\ref{ab}.
\end{proof}

\subsection{Koszul algebras and regularity}

A positively graded $\K$-algebra $R$ is said to be Koszul if its residue field $\K$ has a linear free resolution as an $R$-module. Equivalently, $R$ is Koszul if $\beta^R_{ij}(\K)=0$ for $i\neq j$.  For more details about Koszul algebras, see \cite{Froberg99} and \cite{PP05}.

If $R = E/I$ is a quotient of an exterior algebra $E$ by some graded ideal $I$, $R$ is said to be G-quadratic if $I$ has a Gr\"obner basis consisting of quadrics (homogeneous elements of degree 2) with respect to some coordinates on $E_1$ and some monomial ordering on $E$.  A result of Fr\"oberg \cite{Froberg75} shows that if $I$ is a monomial ideal, then $R$ is Koszul.  It follows from a standard deformation argument \cite{Froberg99} that if $R$ is G-quadratic, then $R$ is again Koszul.

Lastly, $R$ is LG-quadratic if there exists a G-quadratic algebra $A$ and a regular sequence of linear forms $\ell_1,\dots,\ell_r$ such that $R\cong A/(\ell_1,\dots,\ell_r)$.  LG-quadratic algebras were studied in the commutative setting by Caviglia \cite{Caviglia04}, Avramov, Conca, and Iyengar \cite{ACI10} and others.  
We prove in the following section that LG-quadratic quotients of exterior algebras are Koszul; it is immediate from the definition that all G-quadratic algebras are LG-quadratic.  That the  G-quadratic, LG-quadratic, and Koszul quotients of exterior algebras are distinct classes constitutes the bulk of this paper.  For related work on Koszul algebras and exterior algebras, see \cite{Thieu13}, \cite{Thieu20}, \cite{KK08}, and \cite{KK12}.

\section{LG-quadratic algebras}\label{lg}

In this section we show that the Koszul property is preserved when killing a regular element.  As a consequence, we show that LG-quadratic quotients of exterior algebras are Koszul.

We first recall the following change of rings spectral sequence.

\begin{thm}[{\cite[Theorem 10.73]{Rotman09}}] Let $R\rightarrow S$ be a graded map of graded $\K$-algebras, let $M$ be a graded $R$-module, and let $N$ be a graded $S$-module. Then there exists a spectral sequence with second page \[(E_2)_{p,q}=\tor^S_p(\tor^R_q(M,S),N)\] that converges to (as the total homology of a certain double complex) \[\tor^R_{p+q}(M,N).\] \end{thm}

We show that, just as in the commutative case, LG-quadratic algebras are Koszul.

\begin{thm}\label{LG} Let $E$ be an exterior algebra over $\K$, let $I\subset E$ be an ideal, and let $R=E/I$. If $\ell\in R_1$ is regular, then $R$ is Koszul if and only if $\overline{R}=R/(\ell)$ is Koszul.  In particular, if $R$ is LG-quadratic, then $R$ is Koszul.\end{thm}

\begin{proof} Apply the change of rings spectral sequence to the canonical quotient map $R\rightarrow \overline{R}$, where $M=N=\K$.

We can simplify the terms on the second page as follows. 

First, we can compute $\tor^R_q(\K,\overline{R})$ by taking the $R$-free resolution \[\dots \rightarrow R(-2)\xrightarrow{l\cdot} R(-1)\xrightarrow{l\cdot} R\rightarrow \overline{R}\rightarrow 0.\] Tensoring with $\K$, the maps become zero, so the $q$th homology is just $\K(-q)$.

Let $F_\bullet \rightarrow \K$ be a minimal graded free resolution of $\K$ over $\overline{R}$, and let $F_i=\bigoplus_{j\in \Z}\overline{R}(-j)^{\beta_{ij}}$. Twisting preserves exactness, so $F_\bullet(-q)\rightarrow \K(-q)$ is also a minimal graded free resolution over $\overline{R}$. Tensoring with $\K$, the maps become zero, so the $p$th homology is \[\tor^{\overline{R}}_p(\K(-q),\K)=\bigoplus_{j\in \Z}\K(-q-j)^{\beta_{p,j}}.\]

Suppose $\overline{R}$ is Koszul. Then the bigraded Betti numbers $\beta_{i,j}=0$ whenever $i\neq j$, so $(E_2)_{p,q}=\K(-p-q)^{\beta_{p,p}}$. The differential goes right two places and down one, i.e. from $\K(-p-q)$ to $\K(-p-q-1)$. Since the differentials are degree-preserving, they must therefore be zero. Thus, the spectral sequence collapses at the second page. Since there is a filtration of $\tor^R_{p+q}(M,N)$ given by the terms $(E_\infty)_{p,q}$, we have \[\dim_\K\tor^R_{p+q}(M,N)_j=\sum_{i=0}^{p+q}\dim_\K\K(-p-q)^{\beta{p,p}}_j,\] which is zero whenever $j\neq p+q$. Hence $R$ is Koszul.

Now suppose $\overline{R}$ is not Koszul. Then there exist numbers $p,j\geq 0$ such that $\beta_{p,j}\neq 0$ and $p\neq j$. Since $P_\bullet$ is minimal and graded, we must have $j>p$. Assume $p$ is minimal. Then the differential on page two of the spectral sequence with target $(E_2)_{p,0}$ comes from $\K(-p+1)^{\beta_{p,p}}$. By grading, that differential must be zero. Likewise, the differential on page two whose source is $(E_2)_{p,0}$ has $(E_2)_{p+2,-1}=0$ as its target, so that map must also be zero. So this term of the spectral sequence has stabilized, and the $(E_\infty)_{p-i,i}$-filtration of $\tor^R_p(\K,\K)$ gives us \[\dim_\K\tor^R_p(\K,\K)_j\geq \dim_\K\tor^{\overline{R}}_p(\K,\K)_j\geq 1.\] This proves that $R$ is not Koszul either. \end{proof}

A very similar theorem is proved by Conca, De Negri, and Rossi \cite[Theorem 3.2]{CDR13} in the commutative setting.

\section{Depth of exterior edge ideals and LG-quadratic algebras}\label{depth}

Our first task is to construct an LG-quadratic quotient of an exterior algebra that is not G-quadratic.  A proposed example is given by Thieu \cite[Example 5.2.2(ii)]{Thieu13}, however we show in Example~\ref{ThieuExample} that while Thieu's example has no quadratic Gr\"obner basis in a fixed system of coordinates, it is a monomial ideal after a change of coordinates and hence G-quadratic afterall.  We construct a different example as a quotient of an algebra defined by the edge ideal of a graph over an exterior algebra.  We begin with notation on edge ideals of graphs over the exterior algebra.

Let $G = (V,F)$ denote a finite, simple graph on vertex set $V = \{1,\ldots,n\}$ with edge set $F$.  Let $E = \bigwedge_\K \langle e_1,\ldots,e_n\rangle$ denote the corresponding exterior algebra and let $S = \K[x_1,\ldots,x_n]$ the corresponding symmetric algebra.  We define the symmetric and exterior edge ideals associated to $G$ as follows:
\[I_S(G) = \left(\left\{x_ix_j\,|\,\{i,j\} \in F\right\}\right) \subseteq S, \quad   \quad I_E(G) = \left(\left\{e_ie_j\,|\,\{i,j\} \in F\right\}\right) \subseteq E.\]
Edge ideals over polynomial rings have been extensively studied (see e.g. \cite{MV12}, \cite{VanTuyl13}, or \cite{DHS13}) while those over exterior algebras have received less attention; see e.g. \cite{HOEFEL12}.  More generally, monomial ideals over $E$ were studied by Aramova, Herzog, and Hibi \cite{AHH97} and by Aramova, Avramov, and Herzog \cite{AAH00}.  
For terminology on graphs, we refer the reader to \cite{BM76}.

First we collect two basic lemmas comparing the singular varieties of modules in extensions of the underlying exterior algebras.




\begin{lem}\label{lemNewVar} Let $E=\bigwedge_\K\langle e_1,\ldots,e_{n} \rangle$ denote the exterior algebra on $n$ variables and let $M$ be an $E$-module. If $E'=E\otimes_\K \bigwedge_\K\langle x\rangle$ and $M'=M\otimes_E E'$, then $V_{E'}(M') = V_{E}(M)$.
\end{lem}

\begin{proof} Note that as $E$-modules, $E' = E \oplus Ex$ and $M' = M \oplus Mx$.  Let $m' = (a,bx) \in M'$, where $a,b \in M$ are homogeneous with $\deg(b) = \deg(a) - 1$.  Let $\ell' = (\ell, \alpha x) \in E_1'$, where $\ell \in E_1$ and $\alpha \in \K$.  We must show that $\ell'$ is regular on $M'$ if and only if $\ell$ is regular on $M$ or $\alpha \neq 0$.

First suppose $\ell$ is regular on $M$ and  $\ell' m' = 0$.  Then
\[\ell'm' = (\ell a, \ell bx + (-1)^{\deg(a)}\alpha a x) = 0,\]
and both coordinates vanish.  Since $\ell$ is $M$-regular and $\ell a = 0$, we have $a = \ell a'$ for some $a' \in M$ with $\deg(a') = \deg(a) - 1$.  Since 
\[0 = \ell bx + (-1)^{\deg(a)}\alpha a x =  \ell b x + (-1)^{\deg(a)} \alpha \ell a' x = \ell(b + (-1)^{\deg(a)} \alpha a')x, \]
we have $b + (-1)^{\deg(a)}\alpha a' = \ell b'$ for some $b' \in M$.  Therefore
\begin{align*}
    \ell'(a',b'x) &= (\ell a', \ell b'x + \alpha x a') \\
    &= (a, (b + (-1)^{\deg(a)} \alpha a')x + (-1)^{\deg(a')} \alpha a' x) \\
    &= (a, bx) = m'.
\end{align*}
Therefore $\ell'$ is $M'$-regular.


Suppose $\ell'm'=0$ and $\alpha \neq 0$. Then $a=(-1)^{\deg(a)}\alpha^{-1}\ell b$ and it follows that \[\ell'((-1)^{\deg(a)}\alpha^{-1}b,0)=(a,bx).\] Hence $\ell'$ is $M'$-regular.

Conversely, suppose $\ell' = (\ell, \alpha x)$ is $M'$-regular.  If $\alpha \neq 0$, we are done.  Otherwise, suppose $\alpha = 0$ and $\ell a = 0$ with $a \in M$. Then $\ell'(a, 0) = 0$ and so $(a,0) = \ell' (c,dx)$, for some $c,d \in M$.  In particular, $\ell c = a$, showing that $\ell$ is $M$-regular.  The claim then follows.
\end{proof}

\begin{lem}\label{lemKillNewVar} Let $E=\bigwedge_\K\langle e_1,\ldots,e_{n} \rangle$ denote the exterior algebra on $n$ variables and let $M$ be an $E$-module. If $E'=E\otimes_\K \bigwedge_\K\langle x\rangle$ and $M'=M\otimes_E E'/(x)$, then $V_{E'}(M') = \bigcup_i \mathrm{span}_\K\langle A_i \cup \{x\} \rangle$, where $V_{E}(M) = \bigcup_i \mathrm{span}_\K\langle A_i \rangle$.
\end{lem}

\begin{proof} We use that as $E$-modules, $E' = E \oplus Ex$ and $M' = M$.  Let $m \in M$ and $\ell' = (\ell, \alpha x) \in E'_1$, where $\ell \in E_1$ and $\alpha \in \K$.  It is clear that $\ell'm = \ell m$, and thus $\ell'$ is $M$-regular if and only if $\ell$ is $M$-regular.  The claim follows. 
\end{proof}

Next we compute the depth and singular variety of the edge ideal of a path.  The corresponding result over the polynomial ring is well-known.  

\begin{prop}\label{pathprop} Let $P_n$ denote the  path graph on $n$ vertices, and let $E = \bigwedge_\K\langle e_1,\ldots, e_n \rangle$ be the corresponding exterior algebra. 
\begin{enumerate}
    \item If $n \equiv 1\,(\!\!\!\mod 3)$, then $\depth_E(E/I_E(P_n)) = 1$ and
    \[V_E(E/I_E(P_n)) = \bigcup_{i = 0}^{\lfloor n/3 \rfloor} \mathrm{span}_\K\langle e_1,\ldots,\hat{e_{3i+1}},\ldots,e_n \rangle,\]
    where $\hat{e_i}$ denotes that $e_i$ is omitted.  In particular, $\sum_{i = 0}^{\lfloor n/3 \rfloor} e_{3i+1}$ is a regular element on $E/I_E(P_n)$.
    \item If $n \not \equiv 1 \,(\!\!\!\mod 3)$, then $\depth_E(E/I_E(P_n)) = 0$ and $V_E(E/I_E(P_n)) = E_1$.
\end{enumerate}
\end{prop}

\begin{proof} We proceed by induction on $n$.  If $n = 1$, then $I = (0)$, $E/I \cong E = \bigwedge_\K\langle e_1\rangle$, and $V_E(E/I) = E_1$.  If $n = 2$, then $I = (e_1e_2)$ and $E/I = \bigwedge_\K\langle e_1, e_2 \rangle/(e_1e_2)$, so that $xy = 0$ for all $x,y \in E_1$.  Hence there are no regular elements on $E/I$ and $V_E(E/I) = E_1$.  If $n = 3$, then $I = (e_1e_2, e_2e_3) \subseteq \bigwedge_\K\langle e_1,e_2, e_3 \rangle$, and $I:e_2 = (e_1,e_2,e_3)$.  Again there are no regular elements and $V_E(E/I) = E_1$.

Now suppose $n \ge 4$ and that the proposition holds for smaller $n$.  Let $I = I(P_n) = (e_1e_2, e_2e_3,\ldots,e_{n-1}e_n) \subseteq E$.  Let $J = I_E(P_{n-1})E$.  One computes that $J:(e_{n-1}e_n) =  I_
E(P_{n-2})E + (e_{n-2},e_{n-1},e_n)$.  We have the following short exact sequence of $E$-modules:
\[0 \to \frac{E}{J:(e_{n-1}e_n)} \to \frac{E}{J} \to \frac{E}{I} \to 0.\]
Note that $E/J = \frac{\bigwedge_\K\langle e_1,\ldots,e_{n-1} \rangle}{I_E(P_{n-1})} \otimes_{\K} \bigwedge_\K\langle e_n \rangle$ and 
\[E/(J:(e_{n-1}e_n)) = \frac{\bigwedge_\K\langle e_1,\ldots,e_{n-3} \rangle}{I_E(P_{n-3})} \otimes_{\K} \frac{\bigwedge_\K\langle e_{n-2}, e_{n-1}, e_n \rangle}{(e_{n-2}, e_{n-1}, e_n)}.\]

If $n \equiv 1\,(\!\!\!\mod 3)$, then by induction and Lemma~\ref{lemNewVar},  we have $V_{E}(E/J) = \mathrm{span}_\K \langle e_1,\ldots,e_{n-1}\rangle$.  Similarly, by Lemma~\ref{lemKillNewVar}, \[V_E(E/(J:(e_{n-1}e_n))) = \bigcup_{i = 0}^{\lfloor n/3 \rfloor - 1} \mathrm{span}_\K\langle e_1,\ldots,\hat{e_{3i+1}},\ldots,e_{n-3},e_{n-2},e_{n-1},e_n \rangle.\]
Since none of the components of $V_E(E/J)$ and $V_E(E/(J:(e_{n-1e_n})))$ contain each other, it follows from Theorem~\ref{thmAAH}(2) that $V_E(E/I) = V_E(E/J) \cup V_E(E/(J:(e_{n-1e_n})))$ and the conclusion follows.

If $n \equiv 2\,(\!\!\!\mod 3)$, we use the same short exact sequence as above, but in this case, Lemmas~\ref{lemNewVar} and \ref{lemKillNewVar} combined with the induction hypothesis imply that $V_E(E/(J:(e_{n-1e_n}))) = E_1$, while $V_E(E/J) \neq E_1$.  Again, by Theorem~\ref{thmAAH}(4), we have $V_E(E/I) = V_E(E/J) \cup V_E(E/(J:(e_{n-1e_n}))) = E_1$.

Finally, if $n \equiv 0\,(\!\!\!\mod 3)$, we consider the following short exact sequence of $E$-modules:
\[0 \to \frac{E}{I:e_n} \to \frac{E}{I} \to \frac{E}{I + (e_n)} \to 0.\]
As $E/(I + (e_n) = \bigwedge_\K\langle e_1,\ldots,e_{n-1} \rangle/I_E(P_{n-1}) \otimes_{\K} \bigwedge_\K\langle e_n \rangle / (e_n),$ we have $V_E(E/(I + (e_n)) = E_1$ by induction and Lemma~\ref{lemKillNewVar}.  One computes that \[I:e_n = (e_1e_2,\ldots,e_{n-3}e_{n-2},e_{n-1},e_n),\] and so 
\[E/(I:e_n) = \bigwedge_\K\langle e_1,\ldots,e_{n-2} \rangle/I_E(P_{n-2}) \otimes_{\K} \bigwedge_\K\langle e_{n-1},e_n \rangle / (e_{n-1},e_n),\]
 and thus $V_E(E/(I:e_n)) \neq E_1$ by Lemma~\ref{lemKillNewVar} and the inductive hypothesis.  Finally by Theorem~\ref{thmAAH}(2), we get $V_E(E/I) = E_1$.
\end{proof}

If we consider the commutative edge ideal $I_S(P_n)$ in a polynoial ring $S = \K[x_1,\ldots,x_n]$, then $\depth_S(S/I_S(P_n))) = \left \lceil \frac{n}{3} \right \rceil$ (see e.g. \cite[Corollary 5.10]{DS13}), and we see that $\depth_S(S/I_S(G)) - \depth_E(E/I_E(G))$ can be arbitrarily large.  

The previous result will be used to construct an LG-quadratic ideal that is not G-quadratic.  First, we need the following result on Hilbert functions, which could be checked by brute force, computing the Hilbert series of $E/I_E(G)$ for each graph with $6$ vertices and $6$ edges.  We give a more conceptual argument.

\begin{lem}\label{graphlemma} There is no graph $G$ such that $\HS_{E/I_E(G)}(t) = 1+6t+9t^2+t^3$. \end{lem}

\begin{proof} The $n$th coefficient on the Hilbert series of a graph is the number of independent sets of $G$ of size $n$.  So if $G$ is a graph with such a Hilbert series, then it must have $6$ vertices, $\binom{6}{2}-9=6$ edges, and precisely one set of three vertices such that no two of them are adjacent.

Suppose we have a graph $G$ on $6$ vertices where $\{a,b,c\}$ is the lone independent set.

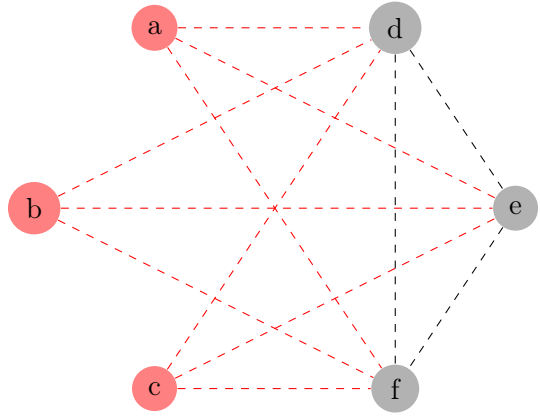
\begin{center}
\begin{tikzpicture}
  [scale=.8,auto=left,node/.style={circle,fill=black!30},rednode/.style={circle,fill=red!50}]
\node[rednode] (a) at (2,6) {a};
\node[rednode] (b) at (0,3) {b};
\node[rednode] (c) at (2,0) {c};
\node[node] (d) at (6,6) {d};
\node[node] (e) at (8,3) {e};
\node[node] (f) at (6,0) {f};

\foreach \from/\to in {a/d,a/e,a/f,b/d,b/e,b/f,c/d,c/e,c/f}
    \draw[dashed,red] (\from) -- (\to);
    
\foreach \from/\to in {d/e,d/f,e/f}
    \draw[dashed] (\from) -- (\to);

\end{tikzpicture}
\end{center}

\noindent That leaves $9$ potential edges with one vertex in $\{a,b,c\}$ and the other in $\{d,e,f\}$ (colored red), and three with both vertices in $\{d,e,f\}$ (colored black). Since we want $G$ to have $6$ edges, we have to eliminate $6$ of the potential edges. 

However, we can eliminate no more than three of the red edges; otherwise, two of the eliminated edges would share a vertex in $\{d,e,f\}$, creating another independent set (the shared vertex and the two from $\{a,b,c\}$ corresponding to the other endpoints of those two edges).

But then we would be forced to eliminate all three black edges, making $\{d,e,f\}$ a second independent set. In any case, it is impossible for $G$ to have precisely one independent set of size $3$. \end{proof}

We can now construct our example as a quotient of an edge ideal algebra.

\begin{thm}\label{LGnotG} There exist  LG-quadratic quotients of exterior algebras that are not G-quadratic.
\end{thm}


\begin{proof}
Consider the graph $P_7$, a path of length $7$, with corresponding edge ideal 
\[I_E(P_7) = (e_1e_2, e_2e_3, e_3e_4, e_4e_5, e_5e_6, e_6e_7).\]
A quick computation shows that 
\[\HS_{E/I_E(P_7)}(t) = 1 + 7t + 15t^2 + 10t^3 + t^4 = (1 + 6t+9t^2+t^3)(1+t).\]  (The first three terms are clear.  That $\HF(E/I_E(P_7),3) = 10$  follows from the fact that there are 10 independent sets of size $3$ in $P_7$ and $\HF(E/I_E(P_7),4) = 1$ since there is one independent set $\{x_1, x_3, x_5, x_7\}$ of size $4$.)  By Proposition~\ref{pathprop}, $\depth_E(E/I_E(P_7)) = 1$ and $e_1 + e_4 + e_7$ is a regular element on $E/I_E(P_7)$.  Thus by definition, \[R = \frac{E}{I_E(P_7) + (e_1 + e_4 + e_7)} \cong \frac{\bigwedge_\K \langle e_1,e_2,e_3,e_4,e_5,e_6\rangle}{(e_1e_2, e_2e_3, e_3e_4, e_4e_5, e_5e_6, e_6(e_1 + e_4))}\]
is LG-quadratic and hence Koszul by Theorem~\ref{LG}.  Moreover, $\HS_R(t) = 1 + 6t + 9t^2 + t^3.$  If $R$ were $G$-quadratic, there would be a quadratic monomial ideal $J$ corresponding to a graph on $6$ vertices with the same Hilbert series, but this is impossible by Lemma~\ref{graphlemma}.  Therefore $R$ is not $G$-quadratic.
\end{proof}

\begin{rmk}
This construction gives a simpler proof that the algebra $B(0)$ of Shelton and Yuzvinsky from \cite[Example 5.2]{SY97} is Koszul.  There the authors define 
\[B(0) = \bigwedge\!\phantom{.}_\K \langle x_1, x_2, x_3, x_4, x_5, x_6 \rangle / (x_1x_6,x_2x_4,x_3x_5,x_1x_4,x_2x_5,(x_4-x_3)x_6) \]
as a certain deformation of an Orlik-Solomon algebra of a hyperplane arrangement.  They use a complicated Koszul duality argument to show that it is Koszul.  Here we recover this by recognizing $B(0)$ as $R$ from Theorem~\ref{LGnotG} under the identification:
\[e_1 \leftrightarrow -x_3,\quad e_2 \leftrightarrow x_5, \quad e_3 \leftrightarrow x_2, \quad e_4 \leftrightarrow x_4, \quad e_5 \leftrightarrow x_1, \quad e_6 \leftrightarrow x_6.\]
\end{rmk}

\section{A Koszul  quotient of an exterior algebra which is not LG-quadratic}\label{nonLG}

In the case of commutative algebras, Conca \cite{Conca} has shown that not every Koszul algebra is LG-quadratic.  As of this writing, this is the lone example of a Koszul algebra that is not LG-quadratic.

\begin{eg}[{\cite[Example 1.20]{Conca}}] Consider the polynomial ring $\K[a,b,c,d]$ and the ideal $I = (ac, ad, ab-bd, a^2+bc, b^2)$.  That the quotient ring $R = S/I$ is Koszul follows from a filtration argument.  That $R$ is not LG-quadratic follows by showing that no quadratic monomial ideal has the same h-polynomial.  This example negatively answered a question of Caviglia \cite[Question 1.2.6]{Caviglia04}.  See also \cite[Question 6.4]{ACI10} and \cite[Question 2.12]{CDR13}.
\end{eg}

The authors were motivated to  answer the following parallel question:

\begin{qst} Is there a Koszul quotient of an exterior algebra that is not LG-quadratic?
\end{qst}

The purpose of this section is to answer this question in the affirmative.  One distinction between our construction and that of Conca is that our example is non-obstructed -- that is, there  exist quadratic monomial ideals with the same Hilbert function as our example.  Thus it takes significant work to show that the example we construct is not LG-quadratic.  

To prove Koszulness, we take advantage of the following result of Fr\"oberg and L\"ofwall:

\begin{thm}[{\cite[Theorem 10.2]{FL02}}]\label{FLthm}
Let $\K$ be a field of characteristic $0$, let $E = \bigwedge_\K \langle e_1,\ldots,e_n \rangle$, and let $I = (f_1,\ldots,f_t)$ be an ideal generated by $t$ generic quadrics.  If $t \ge \binom{n}{2} - \frac{n^2}{4}$, then $E/I$ is Koszul.
\end{thm}

The primary goal in this section is then to prove the following result.

\begin{thm}\label{KoszulNonLG}
Let $E = \bigwedge_\C \langle e_1,\ldots,e_6 \rangle$.  Consider an ideal $I = (f_1,\ldots,f_6) \subseteq E$ generated by $6$ generic quadrics.  Then $E/I$ is Koszul but not LG-quadratic.
\end{thm}

The proof of this theorem will require several steps, which we now summarize.  That such an ideal is Koszul follows from Theorem~\ref{FLthm}.  We show that all quadrics in such an ideal have rank at least 4 by computing the height of the ideal of $4 \times 4$ minors of a generic alternating matrix.  The Hilbert series for the ideal of $6$ generic quadrics in 6 exterior algebra variables is $1 + 6t + 9t^2 = (1+3t)^2$.  If this ideal defined an LG-quadratic algebra, then since Hilbert functions are preserved when passing to the initial ideal, there would be a graph $G$ on $n \ge 6$ vertices with $6$ edges whose edge ideal would have Hilbert series $(1+3t)^2(1+t)^{n-6}$.  If there were no such graphs, we would be done.  However, there are three such graphs, ignoring isolated vertices.  Finally we must check that for each graph $G$, if $I$ is an ideal with initial ideal $in_{<}(I) = I_E(G)$, then $I$ must contain a rank $2$ quadric.  This step is long and technical since we cannot make assumptions about the monomial order itself.  Since the rank of a quadric can only go down after killing a linear form, we conclude that the ideal of $6$ generic quadrics in $6$ variables cannot be LG-quadratic.  

The idea to consider the ranks of the quadrics defining a $G$-quadratic algebra is inspired by a result of Eisenbud, Reeves, and Totaro \cite[Theorem 19]{ERT94}; however, their result gives a bound on rank in terms of the codimension on a polynomial ideal and the number of variables.  We need a stronger statement in the exterior algebra setting.  First  we prove a lower bound on the rank of quadrics in ideals generated by generic quadrics in an exterior algebra.

\begin{prop}\label{pfaffian} Fix positive integers $n,r,t$.  Let $I = (f_1,\ldots,f_t) \subseteq E = \bigwedge_\C \langle e_1,\ldots,e_n \rangle$ be an ideal generated by $t$ generic quadrics.  Then $\rank(q) \ge 2r$ for all nonzero quadrics $q \in I$ if and only if $t \le \frac{(n - 2r + 1)(n-2r+2)}{2}$.
\end{prop}

\begin{proof} First we identify a quadric $q = \sum_{i < j} \alpha_{i,j} e_i e_j$ with the point $p = [\alpha_{1,2}:\cdots:\alpha_{n-1,n}]$ in the projective space $\mathbb{P} = \mathbb{P}_\C^{\binom{n}{2} - 1}$, so that $q = \underline{e} A \underline{e}^\mathsf{T}$, where $A$ is the alternating $n \times n$ matrix with $(i,j)$ entry $\frac{1}{2} \alpha_{i,j}$ if $i > j$ and $- \frac{1}{2} \alpha_{j,i}$ if $j > i$, and $0$ along the diagonal.  Then $\rank(q) < 2r$ if and only if $p \in V(J)$, where $J$ is the ideal of $(2r) \times (2r)$ minors of a generic alternating $n \times n$ matrix $X$.  It follows from \cite[Corollary 2.6]{BE77} that $\sqrt{J} = \mathrm{Pf}_{2r}(X)$, where $\mathrm{Pf}_{2r}(X)$ denotes the ideal of $(2r) \times (2r)$ Pfaffians of $X$.  By \cite[Corollary 2.5]{JP79}, $\codim(J) = \codim(\mathrm{Pf}_{2t}(X)) = \frac{(n-2r+1)(n-2r+2)}{2}$.

A choice of $t$ generic quadrics corresponds to a choice of $t$ generic points from $\mathbb{P}$, whose span is a linear space $L$ of dimension $t-1$.  Thus if $t \le \frac{(n-2r+1)(n-2r+2)}{2}$, then $L \cap V(J) = \varnothing$.  The result follows.
\end{proof}

\begin{rmk} While the radical of the ideal of $(2r)\times(2r)$ minors of a generic alternating matrix is an ideal of Pfaffians, it is not primary to the ideal of Pfaffians.  Rather there are often embedded components.  It would be interesting to compute the full primary decomposition of these ideals.
\end{rmk}

\begin{cor}\label{6quadsrank4} If $I = (f_1,\ldots,f_6) \subseteq E = \bigwedge_\C \langle e_1,\ldots,e_6 \rangle$ an ideal generated by generic quadrics, then every nonzero quadric in $I$ has rank at least $4$.
\end{cor}

\begin{proof} We apply the previous proposition with $n = 6$, $r = 2$ and $t = 6$.
\end{proof}

\begin{cor}\label{2quads4vars} If $q_1, q_2 \in \bigwedge_\C \langle e_1,e_2,e_3,e_4 \rangle$ are linearly independent quadrics, then there is a quadric $q \in (q_1, q_2)$ with $\rank(q) = 2$.
\end{cor}

\begin{proof} We argue as in the proof of Propostion~\ref{pfaffian}.  The subvariety of $\mathbb{P}_\C^5$ corresponding to the space of rank at most $2$ quadrics corresponds to the hypersurface $X = V(x_{1,2}x_{3,4}-x_{1,3}x_{2,4}+x_{1,4}x_{2,3})$.  Two linearly independent quadrics correspond to a line $L$ which must intersect $X$ in at least one point.
\end{proof}

Next we consider edge ideals of graphs that have the right Hilbert series to match an LG-quadratic lift of the ideal of six generic quadrics in six variables.  We will eventually show that none of these can possibly be the initial ideal of an ideal of quadrics of large rank.

\begin{lem}\label{HSgraphs} Let $G$ be a graph on $n$ vertices and let $E = \bigwedge_\C \langle e_1,\ldots,e_n \rangle$ denote the corresponding exterior algebra.
If $\HS_{E/I_E(G)}(t)=(1+3t)^2(1+t)^{n-6}$, then $G$ must be one of the following graphs. The number of isolated vertices in graphs of class (i) is given by $n-5-i$, represented below by discrete nodes to the right.

\vspace{5mm}

(1)

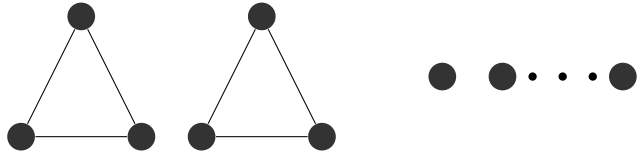
\begin{center}
\begin{tikzpicture}
  [scale=.8,auto=left,node/.style={circle,fill=black!80}]
\node[node] (a) at (0,0) {};
\node[node] (b) at (2,0) {};
\node[node] (c) at (1,2) {};
\node[node] (d) at (3,0) {};
\node[node] (e) at (5,0) {};
\node[node] (f) at (4,2) {};
\node[node] (g) at (7,1) {};
\node[node] (h) at (8,1) {};
\node[node] (i) at (10,1) {};

\draw node[fill,circle,scale=.3] at (9,1) {};
\draw node[fill,circle,scale=.3] at (9.5,1) {};
\draw node[fill,circle,scale=.3] at (8.5,1) {};
\foreach \from/\to in {a/b,a/c,b/c,d/e,d/f,e/f}
    \draw (\from) -- (\to);

\end{tikzpicture}
\end{center}

\vspace{5mm}

(2)

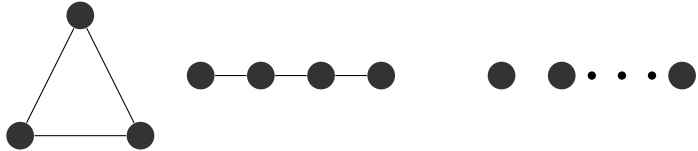
\begin{center}
\begin{tikzpicture}
  [scale=.8,auto=left,node/.style={circle,fill=black!80}]
\node[node] (a) at (0,0) {};
\node[node] (b) at (2,0) {};
\node[node] (c) at (1,2) {};
\node[node] (d) at (3,1) {};
\node[node] (e) at (4,1) {};
\node[node] (f) at (5,1) {};
\node[node] (g) at (6,1) {};
\node[node] (h) at (8,1) {};
\node[node] (i) at (9,1) {};
\node[node] (j) at (11,1) {};

\draw node[fill,circle,scale=.3] at (10,1) {};
\draw node[fill,circle,scale=.3] at (9.5,1) {};
\draw node[fill,circle,scale=.3] at (10.5,1) {};

\foreach \from/\to in {a/b,a/c,b/c,d/e,e/f,f/g}
    \draw (\from) -- (\to);

\end{tikzpicture}
\end{center}

\vspace{5mm}

(3)

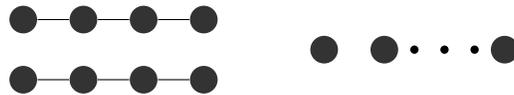
\begin{center}
\begin{tikzpicture}
  [scale=.8,auto=left,node/.style={circle,fill=black!80}]
\node[node] (a) at (0,1.5) {};
\node[node] (b) at (1,1.5) {};
\node[node] (c) at (2,1.5) {};
\node[node] (d) at (3,1.5) {};
\node[node] (e) at (0,0.5) {};
\node[node] (f) at (1,0.5) {};
\node[node] (g) at (2,0.5) {};
\node[node] (h) at (3,0.5) {};
\node[node] (i) at (5,1) {};
\node[node] (j) at (6,1) {};
\node[node] (k) at (8,1) {};

\draw node[fill,circle,scale=.3] at (7,1) {};
\draw node[fill,circle,scale=.3] at (7.5,1) {};
\draw node[fill,circle,scale=.3] at (6.5,1) {};

\foreach \from/\to in {a/b,b/c,c/d,e/f,f/g,g/h}
    \draw (\from) -- (\to);

\end{tikzpicture}
\end{center}
In particular, $G$ has no vertex of degree greater than $2$.
\end{lem}

For the proof, we will need the following lemma.  To simplify notation, we write $\HS(G)$ for $\HS_E(E/I_E(G))$, where $G$ is a graph on $n$ vertices and $E = \bigwedge_\C \langle e_1,\ldots,e_n \rangle.$

\begin{lem}
If $G=G_1\sqcup G_2$, then $\HS(G)=\HS(G_1)\cdot \HS(G_2)$.
\end{lem}

\begin{proof} Let $E_1, E_2$ denote the corresponding exterior algebras for $G_1$ and $G_2$, respectively so that $E = E_1 \tensor_\C E_2$.  The statement follows from the fact that 
$E_1/I_{E_1}(G_1) \tensor_\C E_2/I_{E_2}(G_2) = E/I_E(G)$.
\end{proof}

\begin{proof}[Proof of Lemma~\ref{HSgraphs}.]

Recall that the coefficient on $t^n$ in $\HS(G)$ is the number independent sets of $G$ of size $n$.

Let $n=6$, so $\HS(G)=(1+3t)^2$. If $G$ has more than one connected component, then the Hilbert series of each component must be $1+3t$, which corresponds to a graph on three vertices with three edges, i.e. a triangle. Suppose $G$ is connected. It cannot be a tree, since it has six vertices and six edges, so there must be a cycle. It is easy to eliminate the six graphs with six edges on six vertices with a cycle of length four or greater by calculating their Hilbert series, because they all have independent sets of size $3$. So, there exists a $3$-cycle in $G$. Let's call it $a_1a_2a_3$, and label the remaining vertices $b,c,d$. Without loss of generality, $bc$ is an edge since there must be an edge among $b,c,d$. There must also be an edge among $a_i,b,d$ for all $i$ since $\{a_i,b,d\}$ is not independent, but the only possibility with two remaining edges is $bd$. The same argument for $a_i,c,d$ shows that $cd$ must be an edge. This completes the proof for this case.

Let $n=7$, so $\HS(G)=(1+3t)^2(1+t)$. If $G$ is connected, then it is a tree on seven vertices. We can classify these trees by girth, i.e., the maximal length of a path in $G$. There is only one such tree of girth $m=2$, obtained by connecting one vertex to all other vertices, but its Hilbert series has the wrong coefficient on $t^4$. For girth at least three, consider the following pictures.

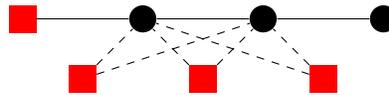
\begin{center}
\begin{tikzpicture}
  [scale=.8,auto=left,node/.style={circle,fill=black!100},rednode/.style={regular polygon, regular polygon sides=4,fill=red!100}]
\node[rednode] (a) at (0,1) {};
\node[node] (b) at (2,1) {};
\node[node] (c) at (4,1) {};
\node[node] (d) at (6,1) {};
\node[rednode] (e) at (1,0) {};
\node[rednode] (f) at (3,0) {};
\node[rednode] (g) at (5,0) {};

\foreach \from/\to in {a/b,b/c,c/d}
    \draw (\from) -- (\to);
    
\foreach \from/\to in {e/b,e/c,f/b,f/c,g/b,g/c}
    \draw[dashed] (\from) -- (\to);

\end{tikzpicture}
\end{center}

If $m=3$, then there is an induced path of length three as above and the remaining edges must be chosen among the dashed lines. There cannot be an edge among the three vertices in the bottom row since that would produce a path of length at least four. In any case, the four red square vertices form an independent set, which means this graph has the wrong Hilbert series.

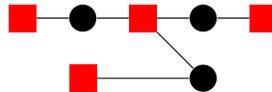
\begin{center}
\begin{tikzpicture}
  [scale=.8,auto=left,node/.style={circle,fill=black!100},rednode/.style={regular polygon, regular polygon sides=4,fill=red!100}]
\node[rednode] (a) at (0,1) {};
\node[node] (b) at (1,1) {};
\node[rednode] (c) at (2,1) {};
\node[node] (d) at (3,1) {};
\node[rednode] (e) at (4,1) {};
\node[rednode] (f) at (1,0) {};
\node[node] (g) at (3,0) {};

\foreach \from/\to in {a/b,b/c,c/d,d/e,c/g,f/g}
    \draw (\from) -- (\to);

\end{tikzpicture}
\end{center}

If $m=4$ and the two remaining vertices are connected, then the above graph is the only possibility, and its Hilbert series has the wrong coefficient on $t^4$ as shown by the red square vertices.

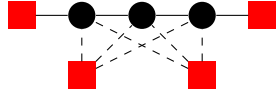
\begin{center}
\begin{tikzpicture}
  [scale=.8,auto=left,node/.style={circle,fill=black!100},rednode/.style={regular polygon, regular polygon sides=4,fill=red!100}]
\node[rednode] (a) at (0,1) {};
\node[node] (b) at (1,1) {};
\node[node] (c) at (2,1) {};
\node[node] (d) at (3,1) {};
\node[rednode] (e) at (4,1) {};
\node[rednode] (f) at (1,0) {};
\node[rednode] (g) at (3,0) {};

\foreach \from/\to in {a/b,b/c,c/d,d/e}
    \draw (\from) -- (\to);
    
\foreach \from/\to in {f/b,f/c,f/d,g/b,g/c,g/d}
    \draw[dashed] (\from) -- (\to);

\end{tikzpicture}
\end{center}

If $m=4$ and the two remaining vertices are not connected, then the rest of the edges must be chosen from the dashed lines above, and in any case the four red square vertices will have no edges between them.

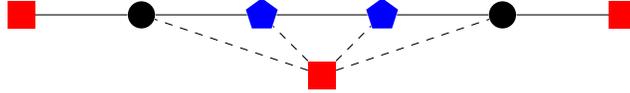
\begin{center}
\begin{tikzpicture}
  [scale=.8,auto=left,node/.style={circle,fill=black!100},rednode/.style={regular polygon, regular polygon sides=4,fill=red!100},yellownode/.style={regular polygon, regular polygon sides=5,fill=blue!100}]
\node[rednode] (a) at (0,1){};
\node[node] (b) at (2,1) {};
\node[yellownode] (c) at (4,1) {};
\node[yellownode] (d) at (6,1) {};
\node[node] (e) at (8,1) {};
\node[rednode] (f) at (10,1) {};
\node[rednode] (g) at (5,0) {};

\foreach \from/\to in {a/b,b/c,c/d,d/e,e/f}
    \draw (\from) -- (\to);
    
\foreach \from/\to in {g/b,g/c,g/d,g/e}
    \draw[dashed] (\from) -- (\to);

\end{tikzpicture}
\end{center}

If $m=5$, the remaining edge must be one of the dashed lines, but then the three red square vertices and one of the green pentagon vertices will form an independent set of size $4$.

If $m=6$, then $G$ is a simple path on seven vertices $a-b-c-d-e-f-g$, and $a,c,e,g$ is an independent set.

Thus, $G$ is not connected, i.e. $G=G_1\sqcup G_2$. It is worth noting that if any graph has Hilbert series $(1+t)^k$, then that graph is $k$ isolated vertices (with no edges). This allows us to perform induction if we find a connected component with such a Hilbert series. The only way to factor $(1+3t)^2(1+t)$ into two polynomials such that neither is $1+t$ is $\HS(G_1)=1+3t$ and $\HS(G_2)=(1+3t)(1+t)$. Then $G_1$ is a triangle as shown above, and $G_2$ is a graph on four vertices with three edges. If we assume $G$ has no isolated vertices (which would reduce to the case $n=6$), then $G_2$ is connected, hence a tree. There are only two trees on four vertices, and the only one with the right Hilbert series is the simple path $a-b-c-d$.

Let $n=8$. The Hilbert series indicates that $G$ has eight vertices and six edges, so it must have at least two connected components $G=G_1\sqcup G_2$. The only way to factor $\HS(G)=(1+3t)^2(1+t)^2$ into two polynomials so that neither is a power of $1+t$ (hence isolated vertices and reduction to a case on fewer vertices) is $\HS(G_i)=(1+3t)(1+t)$ for $i=1,2$. But then, as shown above, $G_i$ either has isolated vertices or is the simple path on four vertices.

Let $n\geq 9$. Then $G$ has at least nine vertices and exactly six edges, which implies at least three connected components. There is no way to factor $(1+3t)^2(1+t)^{n-6}$ into three polynomials such that none is a power of $1+t$, so this reduces to one of the earlier cases.
\end{proof}

As above, to a quadratic form $q \in E = \bigwedge_\C \langle e_1,\ldots,e_n \rangle$, we associate an $n \times n$ alternating matrix $A$ so that
\[q = \underline{e} A \underline{e}^\mathsf{T},\]
where $\underline{e} = (e_1,\ldots,e_n)$.  The rank of $q$ is then $\rank(q) = \rank(A)$.  It is well-known that the rank of an alternating matrix must be even \cite[Corollary 1, p. 351]{Jacobson85i}, say $\rank(q) = 2r$, and that after a change of variables, we can write 
\[q = \sum_{i = 1}^r e_{2r-1}e_{2r}.\]
As we are interested in Gr\"obner bases, which are sensitive to changes of coordinates, we need the following lemma.  Note that when $\ell \in E_1$, we write $e_i > \ell$ if $e_i > e_j$ for all variables $e_j \in \supp(\ell)$.

\begin{lem}\label{rank} Let $q \in E = \bigwedge_\C \langle e_1,\ldots,e_n \rangle$ be a quadric of rank $2r$ and fix a monomial order $<$ on $E$.  Then
\[q = \alpha_1(e_{i_1} + \ell_{1,1})(e_{j_1} + \ell_{1,2}) + \cdots + \alpha_r(e_{i_r} + \ell_{r,1})(e_{j_r} + \ell_{r,2}),\]
where $\alpha_1,\ldots,\alpha_r \in \C^\ast$, $e_{i_1},\ldots,e_{i_r},e_{j_1},\ldots,e_{j_r}$ are distinct variables, $\ell_{s,t} \in E_1$  with support disjoint from $\{e_{i_1},\ldots,e_{i_s},e_{j_1},\ldots,e_{j_t}\}$  and with $e_{i_s} > \ell_{s,1}$ and $e_{j_s} > \ell_{s,2}$ for all $1 \le s \le r$, and
\[e_{i_1}e_{j_1} > e_{i_2}e_{j_2} > \cdots > e_{i_r}e_{j_r}.\]
\end{lem}

\begin{proof} First note that any such $q$ as above has rank $2r$.  We proceed by induction on $r$.  Let $e_{i_1}e_{j_1}$ be the initial monomial of a rank $2r$ quadric $q$ so that $q = \alpha_1 e_{i_1}e_{j_1} + q'$, with $\alpha_1 \in \C^\ast$ and all terms in $q'$ smaller than $e_{i_1}e_{j_1}$.  By grouping terms divisible by either $e_{i_1}$ or $e_{j_1}$, we can rewrite this as 
\[q = \alpha_1(e_{i_1} + \ell_{1,1})(e_{j_1} + \ell_{2,1})  + q'',\]
where all variables in the supports of $\ell_{1,1}$ and $\ell_{1,2}$ are smaller than $e_{i_1}$ and $e_{j_1}$, respectively.  In particular, the variables $e_{i_1}$ and $e_{j_1}$ are not in the support of $q''$ and so $\rank(q'') = \rank(q) - 2$.  Applying the inductive hypothesis to $q''$, the result follows.
\end{proof}

The most technical part of the proof is showing that any ideal whose initial ideal matched the edge ideals from  Lemma~\ref{HSgraphs} must contain a rank $2$ quadric, which is precisely the following lemma.

There is one subtle point that deserves mentioning; when writing the rank $4$ quadric $q = e_1e_2 + e_1e_3 + e_2e_4$ in the form of the previous lemma, we have $q = (e_1-e_4)(e_2+e_3) - e_3e_4$, but $e_3e_4$ is not in the support of $q$.  Care must be taken with the arguments when this occurs.  However, we may sometimes avoid this problem.  Given a quadric $q$ as in the previous Lemma, since $e_{i_1} > \ell_{1,1}$, we can make the change of variables $e_{i_1} \mapsto e_{i_1} - \ell_{1,1}$ without changing the initial ideal.  Continuing in this way, we may always assume that one quadric in a Gr\"obner basis has the following simpler form
\[ q = \alpha_1e_{i_1}e_{j_1} + \cdots + \alpha_r e_{i_r}e_{j_r}.\]
After another change of variables we may assume that $\alpha_i = 1$ for all $i$.

\begin{lem}\label{rank2inI} Let $G$ be one of the three graphs from Lemma~\ref{HSgraphs}.  If $I \subseteq$ is a graded ideal and $<$ is a monomial order such that $in_{<}(I) = I_E(G)$, then there is a rank $2$ quadric in $I$.
\end{lem}

\begin{proof} We may assume that $I$ has a reduced Gr\"obner basis of quadrics of rank at least $4$, or else we are done.  \\

Claim 1: $in_<(I) = I_E(G)$, where $G$ is the graph from Case~(1) from Lemma~\ref{HSgraphs}.\\

Let $q$ denote one of the elements of the reduced Gr\"obner basis with leading term $wx$.  By Lemma~\ref{rank} (and remarks afterward), we can write $q = wx + yz + \cdots$, 
where $wx > yz$ are distinct variables, $\ell_1,\ldots,\ell_4 \in E_1$ and $w > \ell_1, x > \ell_2,\ldots,z > \ell_4$.  Multiplying $q$ by $w$ or $x$, we see that $wyz, xyz \in in_<(I)$.  As $yz \in \supp(q)$, we have $yz \notin in_<(I)$.  Therefore either $wy \in in_<(I)$ or $wz \in in_<(I)$.  Similarly $xy \in in_<(I)$ or $xz \in in_<(I)$.  We conclude that each vertex of $G$ with degree at least one actually has degree at least $2$.  The only possibility for $G$ is Case~(1) from Lemma~\ref{HSgraphs}.

At this point we may assume that $in_<(I) = (uv, uw, vw, xy, xz, yz)$.\\

Claim 2: The Gr\"obner basis of $I$ contains no quadric $q$ with $\rank(q) \ge 6$.\\

Suppose $I$ has a rank $\ge 6$ quadric $q$ in its reduced Gr\"obner basis.  After a linear change of variables that does not change the initial ideal, we may assume that $q$ has the form $q = uv + wa + bc + \cdots$, where $u,v,w,a,b,c$ are distinct variables.  Taking multiples of $q$ we see that $uabc, vabc \in in_<(I)$, which forces $ab \in in_<(I)$ or $ac \in in_<(I)$.  Without loss, we may assume $q$ has the form $q = uv + wx + yc + \cdots$, where $c \notin \{u,v,w,x,y,z\}$.  (If $c \in \{u,v,w,x,y,z\}$, the only choice would be $c = z$, in which case $yz$ would be both the initial monomial of a term in the reduced Gr\"obner basis of $I$ and in the support of $q$, a contradiction.) Replacing the other quadrics by their initial terms does not change the Hilbert function and so $I' = (uw, vw, xy, xz, yz, q)$ has the same Hilbert function as $I$.  Taking the initial ideal of $I'$, we see that $I'' := (uv,uw,vw,xy,xz,yz,uxyc,vxyc) \subseteq in_{<}(I')$.   A calculation shows that the Hilbert series of $E'/I'' = 1 + 7t + 15t^2 + 8t^3$, where $E' = \bigwedge_\K \langle c, u, v, w, x, y, z \rangle$, while  $(1 + 3t)^2(1+t) = 1 + 7t + 15t^2 + 9t^3$.  Therefore, the Hilbert Series of $E/in_<(I)$ is too small, and we conclude that the reduced Gr\"obner basis for $I$ consists of rank $4$ quadrics.\\

Claim 3: The reduced Gr\"obner basis of $I$ contains a quadric of the form $xy + za$, where $in_<(I) = (uv, uw, vw, xy,xz, yz)$ and $a \notin \{u,v,w,x,y,z\}$.\\

Suppose this is not the case.  Then each of the six quadrics in the Gr\"obner basis of $I$ has the form $(a + \ell)(b + \ell') + (c + \ell'')(d + \ell''')$, where $\{a,b,c\} = \{u,v,w\}$ and $d \in \{x,y,z\}$ or $\{a,b,c\} = \{x,y,z\}$ and $d \in \{u,v,w\}$.  Without loss, we may assume that $u > v > w$ and $x > y > z$.  One of the quadrics then has the form $(y + \ell)(z + \ell') + (x + \ell'')(d + \ell''')$, with $\ell,\ell',\ell'',\ell''' \in E_1$, $d \in \{u,v,w\}$ and $yz > xd$.  As $x > y > z$, we must have $z > d \ge w$.  Another quadric has the form $(v + m)(w + m') + (u + m'')(f + m''')$, with $m, m', m'', m''' \in E_1$, $f \in \{x,y,z\}$ and $vw > uf$.  This forces $w > f \ge z$, which contradicts that $z > w$.  This finishes the proof of Claim 3.\\

Claim 4: $I$ contains a rank $2$ quadric.\\

By Claim 3, $I$ contains a quadric of the form $q = xy + za$, where $a \notin \{u,v,w,x,y,z\}$.  A second quadric has the form $q' = (x + k)(z + k') + \alpha(y + k'')(d + k''')$, where $k,k',k'',k''' \in E_1$, $\alpha \in \K^\ast$, and $d$ is a variable distinct from $x,y,z$.  Considering $xq - q'a$, the possibly nonzero leading terms are $xk'a, kza, yda, yk'''a, k''da$.  By our assumptions no monomial from $in_<(I)$ can divide any of these and thus they are all zero.  This forces $c, k, k', k'', k''' \in (a)$.  It follows that $q, q'$ are linearly independent quadrics in $\bigwedge_\K\langle x,y,z,a \rangle$.  By Corollary~\ref{2quads4vars}, $I$ contains a rank $2$ quadric.  
\end{proof} 

\begin{eg} The hypotheses of the previous lemma cannot be weakened considerably.  Consider the ideal 
\[I = (e_1e_2+e_3e_4,e_1e_3+e_2e_4,e_2e_3+e_1e_4,e_5e_6+e_7e_8,e_5e_7+e_6e_8,e_6e_7+e_5e_8),\]
in $E = \bigwedge_\K\langle e_1,\ldots,e_8\rangle$.
It is clear that all $6$ generators of $I$ have rank $4$ and that they form a reverse lexicographic Gr\"obner basis of $I$.  The initial ideal $in_{revlex}(I) = I_E(G)$, where $G$ is the graph from Case~(1) in Lemma~\ref{HSgraphs}.  Note that there is a rank $2$ quadric in $I$:
\[(e_1e_2+e_3e_4)+(e_1e_3+e_2e_4) = (e_1+e_4)(e_2-e_3) \in I.\]
This is in some sense a limiting case later in the proof.
\end{eg}

Now that we have all of the ingredients, we can complete the proof of Theorem~\ref{KoszulNonLG}.

\begin{proof}[Proof of Theorem~\ref{KoszulNonLG}.]  That $E/I$ is Koszul follows from Theorem~\ref{FLthm}.
It follows from  \cite[Theorem 10.5]{FL02} that the Hilbert series of $E/I$ is $HS_{E/I}(t) = 1 + 6t + 9t^2$.  Suppose $I$ is LG-quadratic.  Then there is an ideal $J \subseteq E' = \bigwedge_\C\langle y_1,\ldots,y_n \rangle$ equipped with a monomial order $<$ such that $J$ has a quadratic Gr\"obner basis and there is a regular sequence $\underline{\ell} = \ell_1,\ldots,\ell_d$ on $E'/J$ such that $E'/(J + (\underline{\ell})) \cong E/I$.  It follows that $HS_{E'/J}(t) = (1 + 6t + 9t^2)(1+t)^{n-6}$. Then $in_<(J) = I(G)$ for one of the three graphs from Lemma~\ref{HSgraphs}.  By Lemma~\ref{rank2inI}, $J$ contains a rank $2$ quadric.  Since rank can only go down when we kill a regular element, $I$ contains a rank $2$ quadric.  This contradicts Proposition~\ref{6quadsrank4} and the result follows.
\end{proof}

\begin{rmk} In light of this example, it is reasonable to ask if one can construct similar examples of commutative Koszul algebras that are not LG-quadratic.  Indeed, Fr\"oberg and L\"ofwall proved \cite[Theorem 7.1]{FL02} that if $A = \K[x_1,\ldots,x_n]/I$, where $I$ is generated by $t$ generic quadrics, then $A$ is Koszul if and only if $t \le n$ or $t \ge n^2/4 + n/2$.  The first case corresponds to generic complete intersections which are known to be LG-quadratic \cite[Remark 1.19]{Conca}.  The second case is similar to the setting considered here.  One could compute a lower bound on the ranks of the quadrics in $I$ by computing the codimension of the appropriate ideal of minors of a generic symmetric matrix.  However, it is not clear what one can say about the ranks of quadrics appearing in a quadratic Gr\"obner basis can be in a polynomial ring.  The results in \cite{ERT94} depend on the number of variables, which makes arguments relating to the LG-quadratic property difficult.
\end{rmk}

\section{Examples}\label{examples}

In this section we collect examples to compare quadratic ideals over exterior algebras.  In particular, we have the following implications for quadratic quotients of an exterior algebra $E/I$:
\[\text{quadratic GB} \Rightarrow \text{G-quadratic} \Rightarrow \text{LG-quadratic} \Rightarrow \text{Koszul} \Rightarrow \text{quadratic}.\]
The first two implications follow from the definitions.  The third implication is a consequence of Theorem~\ref{LG}.  The last implication is well-known \cite{Froberg99}.

Theorems~\ref{LGnotG}  and \ref{KoszulNonLG} show that the second and third implications are strict, respectively.
The examples below show that the remaining implications are strict as well.  \\

\begin{eg} \textbf{A G-quadratic ideal that has no quadratic Gr\"obner basis in the given coordinate system}\label{ThieuExample}\\
This example is due to Thieu \cite[Example 5.2.2.ii]{Thieu13}, who shows that it is Koszul via a filtration argument and asks whether it is G-quadratic.  We include a shorter proof of Koszulness here by showing that it is G-quadratic.  Fix a field $\K$ with $\mathrm{char}(\K) \neq 2$.  Let $E = \bigwedge_\K \langle e_1, e_2, e_3, e_4\rangle$ be a 4-variable exterior algebra and consider $I = (e_1e_2 - e_3e_4, e_1e_3 - e_2e_4)$.  First we show that $I$ has no quadratic Gr\"obner basis with respect to any monomial order.  Indeed, fix a monomial order $\prec$ and observe that the initial terms of the generators of $I$ with respect to $\prec$ are  $ (e_ie_j, e_ie_k)$ for some distinct integers $1 \le i,j,k \le 4$.  It is clear that $I_3 = (e_1,e_2,e_3,e_4)^3$, while not every degree $3$ monomial in $E$ is in $(e_ie_j, e_ie_k)$.  Therefore $I$ has no quadratic Gr\"obner basis.  However, since the $\mathrm{char}(\K) \neq 2$, we have 
\begin{linenomath}\begin{align*}
    I &= ((e_1e_2-e_3e_4)+(e_1e_3-e_2e_4), (e_1e_2-e_3e_4)-(e_1e_3-e_2e_4))\\
    &= ((e_1+e_4)(e_2+e_3),(e_1-e_4)(e_2-e_3)).
    \end{align*}\end{linenomath}
    Thus after a linear change of variables, $I$ becomes a quadratic monomial ideal, which is obviously G-quadratic.
\end{eg}

That there are quadratic quotients of an exterior algebra that are not Koszul is well known.  This is true even when the defining ideal is principal.

\begin{eg} \textbf{A quadratic principal ideal that is not Koszul} This example is due to Nguyen \cite{Nguyen14} and is derived from \cite{McCullough19}.  We include a proof here for completeness. Let $E = \bigwedge_\C \langle e_1, e_2, e_3, e_4\rangle$ and let $I = (e_1e_2 + e_3e_4)$.  One checks that the Hilbert series of $E/I$ is $\HS_{E/I}(t) = 1 + 4t + 5t^2$.  If $E/I$ were Koszul, then by \cite[Definition-Theorem 1]{Froberg99} its Poincare series would be
\[P_{E/I}(t) = \frac{1}{\HS_{E/I}(-t)} = 1 + 4t + 11t^2 + 24t^3 + 41t^4 + 44t^5 - 29t^6 \cdots,\]
which has a negative coefficient.  Therefore $E/I$ cannot be Koszul.
\end{eg}

\section*{Acknowledgements}

The authors thank Aldo Conca and Hal Schenck for useful conversations.  The first author was supported by NSF
grant DMS-1900792.

\bibliographystyle{siam}
\bibliography{exterior_koszul}

\end{document}